\documentclass[leqno,11pt]{amsart}

\usepackage{graphicx}
\usepackage{psfrag}
\usepackage{latexsym}
\usepackage{amssymb}
\usepackage{amsfonts}
\usepackage{amsmath}
\usepackage{amsthm}
\usepackage{color,enumitem}
\usepackage[margin=1in]{geometry}
\usepackage{mathtools}
\usepackage{bm}

\newtheorem{theorem}{Theorem}
\numberwithin{theorem}{section}

\newtheorem{corollary}[theorem]{Corollary}

\theoremstyle{definition}

\newtheorem{definition}[theorem]{Definition}

\theoremstyle{remark}

\newtheorem{question}[theorem]{Question}
\newtheorem{problem}[theorem]{Problem}

\numberwithin{equation}{section}

\newcommand{\reals}{\mathbb{R}}
\newcommand{\Heis}{\mathbb{H}}

\newcommand{\nats}{\mathbb{N}}
\newcommand{\comps}{\mathbb{C}}

\newcommand{\ovl}{\overline}

\newcommand{\subeq}{\subseteq}

\newcommand{\bslash}{\backslash}

\newcommand{\elinf}{l^{\infty}}

\newcommand{\grad}{\nabla}

\def\XXint#1#2#3{{\setbox0=\hbox{$#1{#2#3}{\int}$}
\vcenter{\hbox{$#2#3$}}\kern-.5\wd0}}

\def\ul#1{\underline{#1}} 
\author{Kevin Wildrick}

\address{K. Wildrick: Mathematisches Institut, Universit\"at Bern, Sidlerstrasse 5, 3012 Bern, Switzerland ({\tt kevin.wildrick@math.unibe.ch})}
\keywords{Poincar\'e-inequality space, Heisenberg group, Sobolev Space, measurable differentiable structure, non-embedding. 2010 \emph{Mathematics subject classification.} Primary: 46E35, 28A78; Secondary: 46E40, 53C17, 30L99}
\thanks{The author wishes to thank Paul Creutz and Nikita Evseev for informing him of their work, and is grateful to Zolt\'an Balogh and Thomas Z\"urcher for helpful discussions.}

\begin{document}

\title{Bochner Partial Derivatives, Cheeger-Kleiner Differentiability, and Non-Embedding}
\date{\today}
\begin{abstract}Among all Poincar\'e inequality spaces, we define the class of \emph{Cheeger fractals}, which includes the sub-Riemannian Heisenberg group. We show that there is no bi-Lipschitz embedding $\iota$ of any Cheeger fractal $X$ into any Banach space $V$ with the following property: there exists a  bounded Euclidean domain $\Omega$ such that for any Lipschitz mapping $f \colon \Omega \to X$, the Bochner partial derivatives of $\iota \circ f$ exist and are integrable. This extends and provides context for an important related result of Creutz and Evseev. 
\end{abstract}
\maketitle
\section{Introduction} 

Let $\Omega \subeq \reals^n$ be a bounded domain in Euclidean space, and let $(X,d)$ be a complete and separable metric measure space. Denote by $\elinf$ the Banach space of bounded sequences in $X$ equipped with the supremum norm. Recall that a \emph{Kuratowski embedding} is a mapping $\kappa \colon X \to \elinf$ of the form 
$$\kappa(x) = (d(x,x_i)-d(x_0,x_i))_{i \in \nats},$$
where $\{x_i\}_{i \in \nats}$ is a countable dense set in $X$. Each Kuratowski embedding is an isometric embedding \cite[Exercise 12.6]{heinonen_lectures_2001}. 

As evidenced by \cite{hajlasz_sobolev_2008}, 
 \cite{wildrick_peano_2009}, \cite{hajlasz_sobolev_2010}, \cite{balogh_weak_2014}, and several other works, a standard approach to defining partial derivatives of a mapping $f \colon \Omega \to X$ was as follows. One chooses a Kuratowski embedding $\kappa$ of $X$ into the Banach space $\elinf$, and considers the partial derivatives of $\kappa \circ f$ defined by integration by parts using a Bochner integral. Let us call these partial derivatives, should they exist as Bochner integrable mappings from $\Omega$ to $\elinf$,  the \emph{Bochner partial derivatives of $f$ with respect to $\kappa$}. 
 
 The flaw in this method was hinted at in \cite{arendt_mapping_2018} and \cite{caamano_sobolev_2021}, and made explicit by Creutz and Evseev \cite{creutz_approach_2021}. Namely, in order for the Bochner partial derivatives of $f$ with respect to $\kappa$ to exist, they must be measurable and essentially separably valued. Creutz and Evseev showed in \cite{creutz_approach_2021} that if $f$ is non-constant, then this is not the case. The key point in their proof is the appearance of ``distance-to-a-point'' functions in a Kuratowski embedding.

Creutz and Evseev propose a modification of this approach based on replacing the Bochner integral with the Gelfand integral \cite{creutz_approach_2021}, \cite{creutz_weak_2023}. 
This approach, which is likely to become the new standard technique, does not require the partial derivatives of $\kappa \circ f$ to be essentially separably valued, allowing for a better existence theory. In particular, the Gefland partial derivatives of $f$ with respect to $\kappa$ exist and are integrable whenever $f$ is Lipschitz. Moreover, the resulting Sobolev spaces coincide with  ``size of the gradient"-based approaches, namely the Newtonian and Reshetnyak spaces (see \cite[Theorem 1.4]{creutz_approach_2021} and the references therein as well as \cite{creutz_weak_2023}
). 

However, there is an unavoidable disadvantage of the approach of Creutz and Evseev. Let $V$ be a Banach space. A measurable mapping $g \colon \Omega \to V$ is essentially separably valued if and only if it is the point-wise limit of measurable simple mappings \cite[Chapter II]{diestel_vector_1977}. Hence, it follows from \cite{creutz_approach_2021} that the Gelfand partial derivatives of a non-constant $f$ with respect to a Kuratowski embedding cannot be approximated by measurable simple mappings. 

An alternate approach to this issue that would preserve the use of the Bochner integeral (and hence approximation by measurable simple mappings) is to seek a replacement for the Kuratowski embedding. Although the Kuratowski embeddings are isometric embeddings, for many purposes a bi-Lipschitz embedding would suffice: 

\begin{question}\label{Q1} Is there a bi-Lipschitz embedding $\iota \colon X \to V$ of $X$ into a Banach space $V$ so that Bochner partial derivatives with respect to $\iota$ of each Lipschitz mapping $f \colon \Omega \to X$ exist? 
\end{question} 

A trivial but important example is when  $X=[-1,1] \subeq \reals$ equipped with the Euclidean metric and Lebesgue measure, and $V$ is the Hilbert space $\reals$.  In this case the answer to Question \ref{Q1} is clearly ``Yes'': the identity mapping provides the desired embedding.  

The purpose of this work is to show that for many interesting metric measure spaces $X$, including the sub-Riemannian Heisenberg group, the answer to the above question is ``No''.  For such spaces, the issue identified by Creutz and Evseev is not a result of using a Kuratowski embedding, but rather a feature of the geometry of the metric measure space itself. This lends further support for the definition of Creutz and Evseev in the context of such spaces. 

In Section 2 we describe the class of spaces for which we negatively answer Question \ref{Q1}. In Sections 3 and 4 below we address Question \ref{Q1} by developing a connection to Cheeger-Kleiner differentiability theory and the corresponding non-embedding results. Section 5 contains some natural open questions arising from this line of inquiry.  

\section{Cheeger Fractals and non-embeddability}\label{sec fractals}
Roughly speaking, the metric measure spaces for which we will negatively answer Question~\ref{Q1} are Poincar\'e-inequality spaces (as defined in \cite{cheeger_differentiability_2009}) with the property that on a set of positive measure, the Hausdorff dimension of a Gromov-Hausdorff tangent space is strictly greater than the dimension of the measurable differentiable structure. We will call such a space a \emph{Cheeger fractal} - see Definition \ref{Cheeger fractal} below. 

\begin{definition}A metric measure space $(X,d,\mu)$ is a \emph{Poincar\'e-inequality space} if $(X,d)$ is complete, the collection of measurable sets for the measure $\mu$ includes the completion of the Borel $\sigma$-algebra, the measure $\mu$ is doubling, and $(X,d,\mu)$ supports a $p$-Poincar\'e inequality for some $p\geq 1$. 
\end{definition} 

The definition of a Poincar\'e-inequality space in \cite{cheeger_differentiability_2009} includes the requirement of quasiconvexity, which follows from the other assumptions (see \cite[Page 430]{cheeger_differentiability_1999} and \cite{korte_geometric_2007}).

The foundation of differentiation theory in metric measure spaces  is the notion of a measurable differentiable structure and the corresponding differentiability result from \cite{cheeger_differentiability_1999}.  
\begin{definition}\label{Cheeger def} [Cheeger] A \emph{measurable differentiable structure} on a metric measure space $(X,d,\mu)$ is a countable collection of \emph{charts} $\{\phi_\alpha \colon U_\alpha \to \reals^{N(\alpha)}\}_{\alpha}$, where 
\begin{enumerate} 
\item $\{U_\alpha\}_{\alpha}$ is a collection of measurable subsets of $X$ (of positive measure) covering $X$ up to a set of measure $0$,
\item $\sup_{\alpha} N(\alpha)<\infty$,
\item each $\phi_{\alpha}$ is a Lipschitz function,
\item for each Lipschitz function $f \colon X \to \reals$ and each $\alpha$,  there is a  Borel measurable function $\grad_\alpha f \colon U_\alpha \to \left(\reals^{N(\alpha)}\right)^*$ so that for $\mu$-almost every $\ul{x} \in U_{\alpha}$ 
\begin{equation}\label{Cheeger def eq}f(x)-f(\ul{x}) = \grad_{\alpha}f (\ul{x}) (\phi_\alpha(x) -\phi_{\alpha}(\ul{x})) + o(d(x,\ul{x}).\end{equation} 
Moreover, if $G  \colon U_\alpha \to \left(\reals^{N(\alpha)}\right)^*$ is another Borel measurable mapping that satisfies \eqref{Cheeger def eq} $\mu$-almost everywhere, then $G= \grad_{\alpha}f$ $\mu$-almost everywhere. 
\end{enumerate} 
\end{definition} 

\begin{theorem}[Cheeger]\label{Cheeger thm}  If $(X,d,\mu)$ is a Poincar\'e inequality space, then there exists a measurable differentiable structure on $X$. 
\end{theorem} 

\begin{definition}\label{Cheeger fractal} A Poincar\'e-inequality space $(X,d,\mu)$ is a \emph{Cheeger fractal} if there is a measurable differentiable structure $\{\phi_\alpha \colon U_\alpha \to \reals^{N(\alpha)}\}_{\alpha}$ on $X$ with the following property: there is a subset $E$ of positive measure contained in the domain $U_{\alpha}$ of a chart so that at each point $\ul{x} \in E$ there is a Gromov-Hausdorff tangent space $X_{\ul{x}}$ having Hausdorff dimension strictly greater than $N(\alpha)$. 
\end{definition} 

We say that a function $f\colon X \to \reals$ is \emph{(Cheeger) differentiable} at $\ul{x} \in X$ if \eqref{Cheeger def eq} holds. 
This notion of differentiability can be extended to Banach space-valued mapping $\iota \colon X \to V$. In this case, the gradient $\grad_\alpha \iota$ takes values in $\left(\reals^{N(\alpha)}\right)^* \otimes V$. The following application to the non-existence of bi-Lipschitz embeddings was proven in \cite{cheeger_differentiability_1999}, see also \cite{wildrick_sharp_2015}:
  
\begin{theorem}[Cheeger]\label{no diff intro} No bi-Lipschitz embedding of a Cheeger fractal $X$ into any Banach space is Cheeger differentiable almost everywhere with respect to any measurable differentiable structure on $X$. 
\end{theorem} 

A highlight in the development of the theory is the following extension, due to Cheeger and Kleiner \cite{cheeger_differentiability_2009}, of Theorem \ref{Cheeger thm} to include Banach-spaced valued mappings, provided that the target $V$ has the \emph{Radon-Nikodym property}: any Lipschitz path $c \colon [0,1] \to V$ is Fr\'echet differentiable almost everywhere.
 
\begin{theorem}[Cheeger-Kleiner]\label{Cheeger-Kleiner} If $X$ is a Poincar\'e-inequality space and $V$ is a Banach space with the Radon-Nikodym property, then any Lipschitz function $\iota \colon X \to V$ is differentiable almost everywhere with respect to any measurable differentiable structure on $X$.
\end{theorem}

The key tool in this work is the following extension of Theorem \ref{Cheeger-Kleiner}. 

\begin{definition}\label{iRNP} Given a mapping $\iota \colon X \to V$, we say that $V$ has the \emph{$\iota$-Radon-Nikodym property} if for any Lipschitz path $c \colon [0,1] \to X$, the composition $\iota \circ c \colon [0,1] \to V$ is differentiable almost everywhere. 
\end{definition} 

\begin{theorem}\label{Variant} Let $X$ be a Poincar\'e-inequality space, let $V$ be a Banach space, and let $\iota\colon X \to V$ be a Lipschitz mapping. If $V$ has the $\iota$-Radon-Nikodym property, then $\iota$ is differentiable almost everywhere with respect to any measurable differentiable structure on $X$. 
\end{theorem} 

A careful analysis of \cite{cheeger_differentiability_2009} shows that the proof of Theorem \ref{Cheeger-Kleiner} in fact also proves Theorem \ref{Variant}. The following non-embedding result now follows immediately from Theorem \ref{no diff intro} and Theorem \ref{Variant}: 
\begin{corollary}\label{non-embedding}
 Let $X$ be a Cheeger fractal and let $V$ be any Banach space. There is no bi-Lipschitz embedding $\iota \colon X \to V$ such that $V$ has the $\iota$-Radon-Nikodym Property. 
\end{corollary}
 
 \section{The main result} 
 
 \begin{theorem}\label{main} If $(X,d,\mu)$ is a Cheeger fractal, then the answer to Question \ref{Q1} is ``No''. 
 \end{theorem} 
 
 \begin{proof}
We assume without loss of generality that $[0,1]^n$ is contained in $\Omega$. Let  $c \colon [0,1] \to X$ be a Lipschitz path in a Cheeger fractal $X$, and let $\iota \colon X \to V$ be a bi-Lipschitz embedding of $X$ into an arbitrary Banach space $V$. We define a Lipschitz mapping $f_c \colon \Omega \to X$ by 
$$f_c(x_1,\hdots,x_n) = \begin{cases} 
				c(x_1) & 0 \leq x_1 \leq 1, \\
				c(0) & x_1 < 0,\\
				c(1) & x_1 >1.\\
				\end{cases}
				$$
Note that if $\iota \circ f_c$ is differentiable almost everywhere on any translate $[0,1] + (0,x_2,\hdots,x_n)$ where $(x_2,\hdots,x_n) \in [0,1]^{n-1}$, then $\iota \circ c$ is differentiable almost everywhere.

It was stated in \cite{arendt_mapping_2018} and proven in \cite[Theorem 4.16]{kreuter_sobolev_2015} that if the Bochner partial derivatives of a mapping $F \colon \Omega \to V$ exist, then $F$ is absolutely continuous \emph{and differentiable almost everywhere} on almost any line in the usual sense.  

Hence, if the Bochner partial derivatives of $\iota \circ f_c$ for any Lipschitz path $c\colon [0,1] \to X$, then $V$ has the $\iota$-Radon-Nikodym property, contradicting Corollary \ref{non-embedding}.  
\end{proof}
 
\section{Sharpness of the main result}
In Theorem \ref{main}, we have shown that for every bi-Lipschitz embedding $\iota \colon X \to V$ of a Cheeger fractal $X$ into a Banach space $V$, there must be at least one Lipschitz function $f \colon \Omega \to X$ so that the Bochner partial derivatives of $f$ with respect to $\iota$ do not exist. 
This is a weaker statement than the corresponding result of Creutz and Evseev regarding any Kuratowski embedding $\kappa$: 
for \textit{every} non-constant function $f\colon \Omega \to X$, the Bochner partial derivatives of $f$ with respect to $\kappa$ do not exist. Hence it is reasonable to ask:

\begin{question}\label{Q2}
 Let $(X,d,\mu)$ be a Cheeger fractal. Is there a bi-Lipschitz embedding $\iota \colon X \to V$ so that there is at least one non-constant function $f \colon \Omega \to X$ so that the Bochner partial derivatives of $f$ with respect to $\iota$ \emph{do} exist?
\end{question}

In this section, we construct a Cheeger fractal for which the answer to Question \ref{Q2} is ``Yes''. 

The Heisenberg group is a non-commutative Lie-group structure on $\comps \times \reals$ given by the group operation
$$(z,t) \star (z',t') = (z+z',t+t'-2\text{Im}(z\ovl{z'})).$$

For an introduction to the Heisenberg group as a metric measure space, see \cite{capogna_introduction_2007}. We denote by $\Heis$ first Heisenberg group equipped with the Carnot-Carath\'eodory metric $d$ and 3-dimensional Lebesgue measure $\mathcal{L}^3$.  It has been shown that $\Heis$ is a Poincar\'e inequality space (see \cite[Chapter 11.3]{hajlasz_sobolev_1995} and the references therein). Moreover, it follows from Pansu's differentiability theorem \cite{pansu_metriques_1989} that $\Heis$ is a Cheeger fractal: $\Heis$ has a measurable differentiable structure with a single chart $\pi\colon \Heis \to \comps$ given by $\pi(z,t)=z$. This means the differentiable structure is of dimension $2$, while each Gromov-Hausdorff tangent space is Ahlfors-regular of dimension $4$; see \cite{kleiner_differentiable_2016} for more information.

Let $A$ denote $\reals$ equipped with the standard metric $d_\reals$ and Lebesgue measure $\mathcal{L}$. It is well-known that $A$ is a Poincar\'e-inequality space. 

We define the metric measure space $X$ to be the ``supremum'' gluing of $A$ and $\Heis$ along the single point $0 \in \reals \sim (0,0) \in \Heis$, which we denote henceforth by $O$. The construction of $X$ is elementary, and results in the following properties:
\begin{itemize}
\item the restriction of the metric on $X$ to $A \times A$ is $d_\reals$, and the restriction of the metric on $X$ to $\Heis \times \Heis$ is $d$,
\item the distance in $X$ between a point $a \in A \subeq X$ and a point $b \in \Heis \subeq X$ is
$$\max(d_\reals(a,O),d(b,O)),$$ 
\item The measure of a subset $S \subeq X$ is given by
$$\mathcal{L}(S \cap A) + \mathcal{L}^3(S \cap \Heis).$$
\end{itemize}

The techniques of \cite[Chapter 6.14]{heinonen_quasiconformal_1998}  show that $X$ is a Poincar\'e-inequality space. It supports a measurable differentiable structure consisting of two disjoint charts: the identity on $\reals \bslash \{O\}$ and the projection $\pi(z,t)= z$ on $\Heis \bslash \{O\}$.  Since $X$ contains a copy of $\Heis$ as a positive measure subset, it is also a Cheeger fractal.

We now construct an embedding $\iota \colon X \to \elinf$. Let $\kappa$ be a Kuratowski embedding of $\Heis$ into $\elinf$ post-composed with a translation chosen that $\kappa$ maps the origin to the zero sequence. Denote by $\tau \colon \elinf \to \elinf$ the shift operator 
$$\tau(\sigma_1,\sigma_2,\hdots) = (0,\sigma_1,\sigma_2,\hdots).$$

For $x \in X$, define 
$$\iota(x) = \begin{cases} 
			\tau \circ \kappa(x) & x \in \Heis, \\
			(x,0,0,0,\hdots) & x \in \reals.\\ \end{cases}$$
Then $\iota$ is an isometric embedding. 

Let $\Omega=(-1,1) \subeq \reals$. Any Lipschitz mapping $f \colon \Omega \to \reals$ also defines a Lipschitz mapping $F \colon \Omega \to X$, and the Bochner partial derivatives of $F$ with respect to $\iota$ are just the standard partial derivatives of $f$, which exist almost everywhere by the classical Rademacher theorem. 

\section{Directions for further research}

\begin{problem}\label{open 1} Among all Poincar\'e inequality spaces, characterize those for which the answer to Question \ref{Q1} is ``No''. 
\end{problem} 

I expect that the answer to Question \ref{open 1} is rather complicated, as there exist rather pathological Poincar\'e inequality spaces that are not Cheeger fractals (e.g. \cite{laakso_plane_2002}). Also see \cite[Remark 4.64]{cheeger_differentiability_1999}.

In a similar vein:
\begin{problem}\label{open 2} Characterize the metric spaces $(X,d)$ for which there exists a Banach space $V$ and a bi-Lipschitz embedding $\iota \colon X \to V$ so that $V$ has the $\iota$-Radon-Nikodym property. Is there such a metric space that does not bi-Lipschitzly embed into a Banach space that has the Radon-Nikodym property? 
\end{problem}

As mentioned above, the sub-Riemannian Heisenberg group $\Heis$ is a Cheeger fractal with a particularly simple measurable differentiable structure. However, the answer to Question \ref{Q2} in the case of $\Heis$ is still unknown: 

\begin{problem}\label{open 3} Is there \emph{any} non-constant mapping $f \colon \Omega \to \Heis$ and any bi-Lipschitz embedding $\iota \colon \Heis \to V$ into a Banach space so that the Bochner partial deriviatives of $f$ with respect to $\iota$ exist? 
\end{problem} 

We now consider a closely related issue. Consider a Lipschitz path $c \colon [-1,1] \to \Heis$. For each point $t \in [-1,1]$ where $\pi\circ c$ is differentiable, there exists (suitably defined)  element $c'(t)$ of the measurable tangent space to $\Heis$ at $c(t)$, called a \emph{velocity} (see \cite{cheeger_differentiability_2009}). Denote by $\mathcal{V}_p$  the set of velocities in the measurable tangent space to $\Heis$ a point $p \in \Heis$, and consider a bi-Lipschitz embedding $\iota \colon \Heis \to V$. Consider the subset 
$$\mathcal{V}_{\iota,p} =\{v \in \mathcal{V}_p: v=c'(0) \text{ and $\iota \circ c$ is Fréchet differentiable at $0$}\}$$
The methods of \cite{cheeger_differentiability_2009} and the discussion in Section \ref{sec fractals} show that $\dim \mathcal{V}_{\iota,p} \leq 1$ for almost every $p \in \Heis$.  However, it is unclear if $\mathcal{V}_{\iota,p}$ can ever be non-trivial:

\begin{problem}\label{open 4} Is there a bi-Lipschitz embedding $\iota \colon \Heis \to V$ into a Banach space so that $\mathcal{V}_{\iota,p}$ contains a non-zero velocity for some $p \in \Heis$? 
\end{problem}

A simpler version of Problem \ref{open 4}, which can be posed without referring to the tools developed in \cite{cheeger_differentiability_2009}, is as follows:
\begin{problem}\label{open 5} Is there a bi-Lipschitz embedding $\iota \colon \Heis \to V$ of the Heisenberg group into a Banach space $V$ so that there is a single vector $\vec{v} \in V$ satisfying 
$$\iota(x+0i,0)=x\cdot \vec{v}$$
for all $x \in \reals$? 
\end{problem}
\bibliographystyle{acm}
\bibliography{Submission}
\end{document}